\documentclass{amsart} 
\usepackage{amsmath}
\usepackage[mathscr]{eucal} 
\usepackage{amssymb}
\usepackage{latexsym}
\usepackage{amsthm} 
\theoremstyle{plain} 

\newtheorem{theorem}{Theorem\,\!}[]

\newtheorem{theorema}{Theorem A\!\!}
\newtheorem{theoremb}{Theorem B\!\!}
\newtheorem{theoremc}{Theorem C\!\!}

\newtheorem{proposition}{Proposition}

\newtheorem{lemma}{Lemma}  
\newcommand{\supp}{\mathop{\mathrm{supp}}\nolimits} 
 
\newcommand{\card}{\mathop{\mathrm{card}}\nolimits}

\newcommand{\pv}{\mathop{\mathrm{p.v.}}\nolimits} 
\numberwithin{equation}{section}  
\theoremstyle{definition} 
\newtheorem{definition}{Definition}

\newtheorem{assertiontwo}{Proposition $A(j,k)$\!\!}

\theoremstyle{remark}
\newtheorem{remark}{Remark}      
\theoremstyle{sublemma}
\newtheorem{sublemma}{Sublemma}

\def\XXint#1#2#3{{\setbox0=\hbox{$#1{#2#3}{\int}$}
\vcenter{\hbox{$#2#3$}}\kern-.5\wd0}}

\title[oscillatory singular integrals]
{Weighted weak type $(1,1)$ estimates for oscillatory singular integrals 
with Dini  kernels }  
\author{Shuichi Sato} 
 
\begin{document} 

\address{Department of Mathematics, 
Faculty of Education, 
Kanazawa University,    
Kanazawa 920-1192, 
Japan} 
\email{shuichi@kenroku.kanazawa-u.ac.jp} 

\begin{abstract}
We consider $A_1$-weights and prove weighted weak type $(1,1)$ estimates 
 for oscillatory singular integrals with kernels satisfying a Dini condition. 
\end{abstract}  
  \thanks{ 
 \endgraf
   }
\maketitle 

\section{Introduction} \label{s1} 

We consider an oscillatory singular integral  operator of the form: 
$$T(f)(x) = \pv \int_{\text{${\bold R}^n$}} e^{iP(x,y)}K(x - y)f(y)\, dy 
= \lim_{\epsilon \to 0} \int_{|x-y|>\epsilon} e^{iP(x,y)}K(x - y)f(y)\, dy ,  $$
where $P$ is a real-valued polynomial:  
\begin{equation} \label{tag 1.1}   
P(x,y) = \sum_{|\alpha| \leq M, |\beta| \leq N} 
a_{\alpha \beta}x^{\alpha}y^{\beta}, 
\end{equation} 
and $f\in \frak S(\bold R^n)$ (the Schwartz space).    
\par
Let $K \in C^1(\bold R^n \setminus \{0\})$ satisfy 
\begin{equation} \label{tag 1.2}   
|K(x)| \leq c|x|^{-n}, \qquad \qquad |\nabla K(x)| \leq c|x|^{-n-1} ; 
\end{equation} 
\begin{equation}\label{tag 1.3}  
\int_{a<|x|<b}K(x)\, dx = 0  \quad \text{for all $a, b$ with $0<a<b$}. 
\end{equation} 
  The smallest constant for which (1.2) holds will be denoted by $C(K)$.  
The following results are known.
\begin{theorema}[{Ricci-Stein \cite{5}}] Let $1 < p < \infty$.  Then, $T$ is bounded on $L^p(\bold R^n)$ with the operator norm  bounded by a constant depending only on the total degree of $P$,  $C(K)$, $p$  and the dimension $n$.
\end{theorema}

\begin{theoremb}[{Chanillo-Christ \cite{1}}] The operator $T$ is bounded from $L^1(\bold R^n)$ to the 
weak $L^1(\bold R^n)$ space with the operator norm  bounded by a constant depending only on the total degree of $P$, $C(K)$ and the dimension $n$.
\end{theoremb}  
\par 
Let $w$ be a locally integrable positive function on $\bold R^n$.  We say 
that $w \in A_1$ if there is a constant $c$ such that 
\begin{equation}\label{tag 1.4} 
M(w)(x)\leq cw(x) \quad \text{a.e.} 
\end{equation} 
where $M$ denotes the Hardy-Littlewood maximal operator. 
The smallest constant for which (1.4) holds will be denoted by $C_1(w)$. 
\par 
It is known that  $T$ is bounded from $L^1_w$ to $L^{1,\infty}_w$ 
(the weak $L^1_w$ space).     

\begin{theoremc}[{\cite{8}}]  There exists a constant $c$ depending 
only on the total 
degree of $P$, $C(K)$, $C_1(w)$ and the dimension $n$ such that   
$$\sup_{\lambda > 0} \lambda w\left(\left\{x \in \bold R^n : 
\left|T(f)(x)\right| > \lambda \right\}\right) \leq c\|f\|_{L^1_w} ,  $$
where $w(E) = \int_E w(x)\, dx$ and $\|f\|_{L^1_w} = \int |f(x)|w(x)\, dx $. 
\end{theoremc} 
\par 
Let $K$ be  locally integrable away from the origin.   
Put, for $r\geq 1$, $0<t\leq 1$  and $R>0$,        
$$\omega_{r,R}(t)=\sup_{|y|\leq Rt/2}\left(R^{-n}\int\limits_{R\leq 
|x|\leq 2R}\left|R^{n}\left(K(x-y)-K(x)\right)\right|^r\, dx\right)^{1/r}.$$
We say that the kernel $K$ satisfies the  $D_r$-condition if 
\begin{gather*} 
B_r=\int_0^1 \omega_r(t)\, \frac{dt}{t} < \infty \qquad \text{where} \quad 
\omega_r(t)= \sup_{R>0}\omega_{r,R}(t) ;  
\\  
C_r=\sup_{R>0}\left(R^{-n}\int\limits_{R\leq 
|x|\leq 2R}\left|R^{n}K(x)\right|^r\, dx\right)^{1/r} < \infty .    
\end{gather*}  
By the usual modifications we can also define the $D_\infty$-condition. 
In this note we shall prove the following results, which will improve Theorems 
B and C. 
\begin{theorem}  Let $r>1$ and  $1/r + 1/u=1$.  Suppose 
the kernel $K$ satisfy the $D_r$-condition and $(1.3)$, and suppose  
$w^u\in A_1$.
Then, there exists a constant $c$ depending only on the total 
degree of $P$, $B_r$, $C_r$, $C_1(w^u)$, $r$ and the dimension $n$ such that   
$$\sup_{\lambda > 0} \lambda w\left(\left\{x \in \bold R^n : 
\left|T(f)(x)\right| > \lambda \right\}\right) \leq c\|f\|_{L^1_w}.  $$
\end{theorem} 
\begin{theorem}  Suppose that $K$ satisfies the $D_1$-condition 
and $(1.3)$. 
Then, there exists a constant $c$ depending only on the total 
degree of $P$, $B_1$, $C_1$ and the dimension $n$ such that   
$$\sup_{\lambda > 0} \lambda \left|\left\{x \in \bold R^n : 
\left|T(f)(x)\right| > \lambda \right\}\right| \leq c\|f\|_{L^1}.  $$
\end{theorem} 
\par 
Every kernel satisfying (1.2) satisfies the $D_\infty$-condition.  If 
$K(x)=|x|^{-n}\Omega(x')$, $x'=x/|x|$, and if $\Omega$ satisfies the 
$L^r$-Dini condition on $S^{n-1}$, then $K$ satisfies the $D_r$-condition. 
\par      
These theorems will be proved by a double induction as in \cite{5}, \cite{1} 
and \cite{8}.  In this note we shall prove only Theorem 1. Theorem 2 can be 
proved similarly.   
Let $P$ be a polynomial of the form in (1.1). We assume that there exists $\alpha$ such that $|\alpha| = M$ and $a_{\alpha \beta} \ne 0$ for some $\beta$.  We 
write 
\begin{equation}\label{tag 1.5} 
P(x,y) = \sum_{|\alpha| \leq M} x^{\alpha}Q_{\alpha}(y)   
\end{equation} 
and define 
  $L = \max \{\text{deg}(Q_{\alpha}) : Q_{\alpha} \ne 0,  |\alpha| = M \}$.
Then $0 \leq L \leq N$. We assume that $L \geq 1$ and 
$\max_{ |\alpha| = M,  |\beta| = L} |a_{\alpha \beta}| = 1$.     
Under this assumption on a polynomial $P$, we define 
$$T_{\infty}(f)(x) = \int_{|x-y|>1} e^{iP(x,y)}K(x - y)f(y)\, dy. $$
To prove Theorem 1, we shall use  the following result in the 
induction. 

\begin{proposition}  Let $\eta$, $\rho>0$ and let the kernel $K$, the 
weight $w$ 
and the exponents $r$, $u$ be as in Theorem $1$.  Then, 
there exists a constant $c$ depending only on $\eta$, $\rho$, the 
total degree of $P$, $r$  and the dimension $n$ 
such that  if  $C_1(w^u) \leq \eta$, $B_r$, $C_r \leq \rho$,
$$\sup_{\lambda > 0} \lambda w\left(\left\{x \in \bold R^n : 
\left|T_{\infty}(f)(x)\right| > \lambda \right\}\right) \leq c\|f\|_{L^1_w}. $$
\end{proposition} 
\par 
Let  $A(f)(x) =  \pv K*f(x)$.   
We need the following result for the first step of induction for the proof 
of  Theorem 1.      
\begin{proposition}
Let the kernel $K$, the weight $w$ and the exponents $r$, $u$ be as in 
Theorem $1$.  
 Let $\eta$, $\rho>0$. There exists a constant $c$ depending only on $\eta$, 
 $\rho$,  $r$  and the dimension $n$ 
such that  if  $C_1(w^u) \leq \eta$, $B_r$, $C_r\leq \rho$,  then  
 $$\sup_{\lambda > 0} \lambda w\left(\left\{x \in \bold R^n : 
\left|A(f)(x)\right| > \lambda \right\}\right) \leq c\|f\|_{L^1_w}.  $$ 
\end{proposition} 
\par 
Since $A$ is bounded on $L^2$ (see  \cite[pp.\ 25--26]{6}), if $A$ is 
as in Proposition 2,   we see that $A$ is a 
singular integral operator considered in \cite[p.\ 13]{6}. Hence 
the conclusion of Proposition 2 will follow from 
\cite[p.\ 15, Theorem 1.6]{6}.  
\par 
We shall give the outlines of the proofs of Theorem 1 and Proposition 1 in 
Sections  2 and 4, respectively.  
Our proof of Proposition 1 is based on the techniques  in Christ \cite{3} 
for the proofs of the weak $(1, 1)$ estimates for rough operators 
(see also Christ-Rubio \cite{4} and  Sato \cite{7}).  We also use the geometrical 
argument of Chanillo-Christ \cite{1}.    We have to prove a key estimate (Lemma 8 in \S 5) in the unweighted case in order to apply the method of Vargas 
 \cite{9} involving an interpolation with change of measure.      
To prove Lemma 8, we need a geometrical result for polynomials 
(Lemma 6 in \S 5).   
We shall prove Lemma 6 in \S 6 by using the results appearing in the proof 
of Chanillo-Christ \cite[{\sc Lemma} 4.1]{1}. 
Lemmas 6 and 8 have been proved in \cite{8}. We include the proofs and some 
other parts of \cite{8} almost verbatim  for the sake of completeness.

\section{Outline of proof of Theorem 1}  

To apply the induction argument of \cite{5} we need some preparation. 
We may assume that $M \geq 1$ and $N \geq 1$; otherwise Theorem 1 reduces to 
Proposition 2.   
\par 
 We write a polynomial in (1.1) as follows:
$$P(x,y) = \sum_{j=0}^M \sum_{|\alpha| = j} x^{\alpha}Q_{\alpha}(y) =:   
\sum_{j=0}^MP_j(x,y). $$ 
 We further decompose $P_j$ as follows:
$$P_j(x,y) = \sum_{t=0}^N \sum\limits_{\substack{|\alpha| = j \\ |\beta| = t 
}} a_{\alpha \beta}x^{\alpha}y^{\beta} =:  \sum_{t=0}^N P_{jt}(x,y).  $$
  For $j = 1, 2, \dots , M$ and $k = 0, 1, \dots , N$, define 
\begin{equation} \label{tag 2.1} 
R_{jk}(x,y) =  \sum_{s=0}^{j-1}P_s(x,y) + \sum_{t=0}^kP_{jt}(x,y). 
\end{equation} 
Note that $R_{jN} =  \sum_{s=0}^{j}P_s \quad (j = 1,2, \dots , M)$.
\par 
For $j = 1, 2, \dots , M$ and $k = 0, 1, \dots , N$,  we consider the 
following propositions.   
\begin{assertiontwo}  
Let  $\eta$, $\rho > 0$.   There exists a constant $c$ depending only on 
$\eta$, $\rho$, $j$, 
$N$, $r$ and the dimension $n$ such that if  $C_1(w^u) \leq \eta$, 
$B_r$, $C_r\leq \rho$   
and if $R_{jk}$  is a polynomial of the form in $(2.1)$, then 
$$\sup_{\lambda > 0} \lambda w\left(\left\{x \in \bold R^n : 
\left|T_{jk}(f)(x)\right| > \lambda \right\}\right) \leq c\|f\|_{L^1_w} ,  $$
where 
$$T_{jk}(f)(x) =  \pv \int_{\text{${\bold R}^n$}} 
e^{iR_{jk}(x,y)}K(x - y)f(y)\, dy. $$
\end{assertiontwo}  

Then, Theorem 1 follows from Proposition $A(M,N)$.  
We shall prove it by  double induction.  We first note that $A(1,0)$ follows 
from the boundedness of the operator $A$.  

Next,  we observe that if $M \geq 2$ and if $A(j,N)$ ($1 \leq j \leq M - 1$) 
is true,  so is $A(j + 1,0)$ since 
$$R_{j+1,0}(x,y) =  R_{jN}(x,y) + \sum_{|\alpha| = j+1}a_{\alpha 0}x^{\alpha} $$  and hence 
$|T_{j+1,0}(f)(x)| =  |T_{jN}(f)(x)| $.   
Thus, to complete the induction  starting from $A(1,0)$ and arriving at $A(M,N)$, it is sufficient to prove $A(j, k+1)$ assuming $A(j,k)$ ($0 \leq k < N$, 
$1 \leq j \leq M$).  
To achieve this, put $R = R_{j,k+1}$, $R_0 = R_{jk}$, $T_{j,k+1} = S$. We note 
that
$$R(x,y) = R_0(x,y) + \sum\limits_{\substack{|\alpha|=j \\ |\beta| = k+1}} 
 a_{\alpha \beta}x^{\alpha}y^{\beta}.$$   
We have only to deal with  the case $C_{jk} = \max_{|\alpha|=j, |\beta| = k+1} 
|a_{\alpha \beta}| \neq 0$. Then, by a suitable dilation we may assume $C_{jk} = 1$.  This can be seen as follows.  We first note that, for $a > 0$, 
$$S(f)(ax) =  \pv \int e^{iR(ax,ay)}K_a(x - y)f(ay)\, dy , $$  
where $K_a(x) = a^nK(ax)$.  Assume the boundedness of $S$ for the case $C_{jk} = 1$. Then, choosing $a$ to satisfy $a^{j+k+1}C_{jk} = 1$, and using the dilation invariance of 
both the class $A_1$ and the class of the kernels considered in Theorem 1, 
we get
\begin{align*}  
w\left(\left\{x \in \bold R^n : \left|S(f)(x)\right| > \lambda \right\}\right) &= w_a\left(\left\{x \in \bold R^n : \left|S(f)(ax)\right| > 
\lambda \right\}\right) 
\\
&\leq c\lambda^{-1} \int |f(ax)|a^nw(ax)\, dx 
\\
&= c\lambda^{-1}\|f\|_{L^1_w}.  
\end{align*}   
\par 
 We split the kernel $K$ as $K = K_0 + K_{\infty}$, where 
$K_0(x) = K(x)$ if $|x| \leq 1$ and $ K_{\infty}(x) = K(x)$ if $|x| > 1$. Assuming $C_{jk} = 1$,  we consider the corresponding splitting 
$S = S_0 + S_{\infty}$: 
$$S_{0}(f)(x) =  \pv \int  e^{iR(x,y)}K_0(x - y)f(y)\, dy, $$
$$S_{\infty}(f)(x) =  \int e^{iR(x,y)}K_{\infty}(x - y)f(y)\, dy. $$
\par 
In the next section, we shall prove 
\begin{equation}\label{tag 2.2} 
\sup_{\lambda > 0} \lambda w\left(\left\{x \in \bold R^n : 
\left|S_{0}(f)(x)\right| > \lambda \right\}\right) \leq c\|f\|_{L^1_w}, 
\end{equation}
while by Proposition 1  we have 
\begin{equation} \label{tag 2.3} 
\sup_{\lambda > 0} \lambda w\left(\left\{x \in \bold R^n : 
\left|S_{\infty}(f)(x)\right| > \lambda \right\}\right) \leq c\|f\|_{L^1_w}.  
\end{equation} 
Combining (2.2) and (2.3), we shall complete the proof of $A(j,k+1)$, which will finish the proof of Theorem 1.

\section{Estimate for $S_0$} 

In this section, we shall prove, under the assumption made in \S 2, that 
if  $C_1(w) \leq \eta$, $B_r$, $C_r\leq \rho$   ($\eta$, $\rho > 0$), 
then  $S_0$ is bounded from $L^1_w$ to $L^{1,\infty}_w$ with the operator norm bounded by a 
constant depending only on $j$, $N$, $\eta$, $\rho$,  $r$ and $n$ ((2.2)). 
\par 
First, we shall prove 
\begin{equation} \label{tag 3.1} 
 w\left(\left\{x \in B(0,1) : \left|S_{0}(f)(x)\right| > \lambda \right\}\right) \leq c\lambda^{-1}\int_{|y|<2}|f(y)|w(y)\, dy , 
\end{equation} 
where $B(x,r)$ denotes the closed ball with center $x$ and radius $r > 0$.

\begin{lemma} Let $w, w^u (1\leq u<\infty) \in A_1$.   
Let $T$ be an operator of the form$:$ 
$$T(f)(x) = \pv \int_{\text{$\bold R^n$}} K(x,y)f(y)\, dy 
= \lim_{\epsilon \to 0} \int_{|x-y|>\epsilon} K(x,y)f(y)\, dy $$
for $f \in \frak S(\bold R^n)$. 
Let $1/r+1/u=1$ and consider a non-negative function $L$ on 
$\bold R^n\setminus \{0\}$ 
satisfying $J_r<\infty$, where 
$$J_r=\sup_{R>0}\left(R^{-n}\int\limits_{R\leq 
|x|\leq 2R}\left(R^{n}L(x)\right)^r\, dx\right)^{1/r}   $$ 
for $r<\infty$ and $J_\infty$ can be defined by the usual modification. 
Suppose the kernel $K$ satisfies $ |K(x,y)| \leq  L(x - y)$.   
For $\epsilon > 0$, put 
$$T_{\epsilon}(f)(x) = \pv \int_{|x-y|<\epsilon} K(x,y)f(y)\, dy. $$
Suppose 
$$\sup_{\lambda > 0} \lambda w\left(\left\{x \in \bold R^n : 
\left|T(f)(x)\right| > \lambda \right\}\right) \leq c_w\|f\|_{L^1_w} . $$
 Then
$$\sup_{\lambda > 0} \lambda w\left(\left\{x \in \bold R^n : 
\left|T_{\epsilon}(f)(x)\right| > \lambda \right\}\right) \leq 
c(c_w + J_rC_1(w^u)^{1/u}) \|f\|_{L^1_w} . $$
\end{lemma}  

\begin{proof} The proof is similar to that of {\sc Lemma} in \cite[p.\ 187]{5}.   We shall prove 
\begin{multline}  \label{tag 3.2} 
 w\left(\left\{x \in B(h,\epsilon /4) : \left|T_{\epsilon}(f)(x)\right| > \lambda \right\}\right) 
 \\
 \leq c(c_w + J_rC_1(w^u)^{1/u})\lambda^{-1}\int_{|y-h|<6\epsilon/4}
 |f(y)|w(y)\, dy  
 \end{multline} 
uniformly in $h \in \bold R^n$. Integrating both sides of the inequality in 
(3.2) with respect to $h$, we get the conclusion of Lemma 1.

Split $f$ into 3 pieces: $f = f_1 + f_2 + f_3$, where $f_i \in 
\frak S(\bold R^n)$, $|f_i|\leq c|f|$ ($i=1, 2, 3$); 
$\supp(f_1) \subset B(h,\epsilon/2)$, 
$\supp(f_2) \subset B(h,11\epsilon/8)\setminus B(h, 3\epsilon/8)$, 
$\supp(f_3) \subset \{x : |x-h|\geq 5\epsilon/4\}$. 
Note that if $|x - h| \leq \epsilon /4$, then $T_{\epsilon}(f_1)(x) = T(f_1)(x)$; since  $|y - h| \leq \epsilon /2$ and $|x - h| \leq \epsilon /4$ imply $|x - y| < \epsilon$.  So by the assumption on $T$, we have  
$$ 
 w\left(\left\{x \in B(h,\epsilon /4) : \left|T_{\epsilon}(f_1)(x)\right| 
 > \lambda \right\}\right) 
\leq c_w\lambda^{-1}\int_{|y-h|<6\epsilon/4}|f(y)|w(y)\, dy. 
$$
\par 
Next, by Chebyshev's inequality, H\"{o}lder's inequality and the fact 
$w^u\in A_1$ we easily see that  
$$ 
 w\left(\left\{x \in B(h,\epsilon /4) : \left|T_{\epsilon}(f_2)(x)\right| > \lambda \right\}\right) 
\leq cJ_rC_1(w^u)^{1/u}\lambda^{-1}\int_{|y-h|<6\epsilon/4}|f(y)|w(y)\, dy.  
$$

Finally, if $|x - h| \leq \epsilon /4$ and $|y - h| \geq 5\epsilon /4$, then  $|x - y| \geq \epsilon$, and so $T_{\epsilon}(f_3)(x) = 0$.   
Combining these results, we get (3.2). This completes the proof of Lemma 1. 
\end{proof} 
\par 
Now we return to the proof of (3.1). If $ |x| \leq 1 $ and $|y| \leq 2$, then 
$$\left|\exp \left(iR(x,y)\right) - \exp \left(i\left(R_0(x,y) +  
\sum\limits_{\substack{|\alpha| = j \\ |\beta| = k+1}}  a_{\alpha \beta}y^{\alpha + \beta}\right)\right)\right| \leq c|x - y|, $$
where $c$ depends only on $k, j $ and $n$. 
\par 
Hence, if $|x| \leq 1$, 
$$
|S_{0}(f)(x)| \leq \left|U\left(\exp \left(i\sum\limits_{\substack{|\alpha| = j \\ |\beta| = k+1}} a_{\alpha \beta}y^{\alpha + \beta}\right)f(y)\right)(x)
\right| + cI(f)(x), 
$$
where
$$ 
U(f)(x) =  \pv \int  e^{iR_0(x,y)}K_0(x - y)f(y)\, dy, \ 
 I(f)(x) = \int_{|x-y|<1}|x - y|L(x-y)|f(y)|\, dy. 
$$ 
\par 
Note that  $U(f)(x) = U(f\chi_{B(0,2)})(x)$,  $I(f)(x) = I(f\chi_{B(0,2)})(x)$   if $|x| < 1$.  By the induction hypothesis $A(j,k)$ and Lemma 1, we see that $U$ is bounded from $L^1_w$ to $L^{1,\infty}_w$.  On the other hand,  it is 
easy to see that 
$$ 
\int_{|x-y|<1}|x - y|L(x-y)w(x)\, dx  \leq \sum_{j \leq 0}2^j
\int_{2^{j-1}\leq |x-y|\leq 2^j} L(x-y)w(x)\, dx  \leq cJ_rM_u(w)(y),   
$$
where $M_u(w)=M(w^u)^{1/u}$.  Thus, by Chebyshev's inequality and the fact 
$w^u\in A_1$ we have
$$
w(\{x \in B(0,1) : I(f)(x) > \lambda \}) 
\leq c J_rC_1(w^u)^{1/u}\lambda^{-1}\int_{|y|<2}|f(y)|w(y)\, dy. 
$$ 
Combining these results, we get (3.1).
\par 
Similarly we can prove 
\begin{equation}\label{tag 3.3}  
 w\left(\left\{x \in B(h,1) : \left|S_{0}(f)(x)\right| > \lambda \right\}
\right) \leq c\lambda^{-1}\int_{|y-h|<2}|f(y)|w(y)\, dy, 
\end{equation}
where $c$ is independent of $h \in \bold R^n$.  
To see this, we first note that 
$$S_{0}(f)(x+h) =  \pv \int  e^{iR(x+h,y+h)}K_0(x - y)f(y + h)\, dy $$
and 
$$R(x + h,y + h) =R_1(x,y,h) +  \sum\limits_{\substack{|\alpha| = j \\ |\beta| = k+1}} a_{\alpha \beta}x^{\alpha}y^{\beta}. $$

We can apply the induction hypothesis $A(j,k)$ to the operator 
$$ \pv \int  e^{iR_1(x,y,h)}K(x - y)f(y)\, dy $$
to get its boundedness from $L^1_w$ to $L^{1,\infty}_w$.  
Thus, by the same argument that leads to (3.1) we get
\begin{align*} 
 w\left(\left\{x \in B(h,1) : \left|S_{0}(f)(x)\right| > \lambda \right\}\right)&= \tau_hw\left(\left\{x \in B(0,1) : \left|S_{0}(f)(x + h)\right| > \lambda \right\}\right) 
\\
&\leq c\lambda^{-1}\int_{|y|<2}|f(y + h)|w(y + h)\, dy 
\\
&\leq c\lambda^{-1}\int_{|y-h|<2}|f(y)|w(y)\, dy , 
\end{align*}
where $\tau_hw(x) = w(x + h)$, and we have used the translation invariance of the class $A_1$. 
Integrating both sides of the inequality (3.3) with respect to $h$, 
we get (2.2).

\section{Outline of proof of Proposition $1$}

Let $f\in \frak S(\Bbb R^n)$. By Calder\'{o}n-Zygmund 
decomposition at height $\lambda > 0 $  we have  a collection $\{Q\}$ of 
non-overlapping closed dyadic cubes and functions $ g, b$ such that 

\begin{gather}  
 f = g + b  ;  
 \\ 
\lambda \leq |Q|^{-1}\int_Q|f| \leq c\lambda  ; 
\\ 
v(\cup Q ) \leq c_v\|f\|_{L^1_v}/\lambda \qquad \text{for all $v \in A_1$} ; 
\\ 
\|g\|_{\infty} \leq c\lambda, \qquad  \|g\|_{L^1_v} \leq c_v\|f\|_{L^1_v}  
\qquad \text{for all $v \in A_1$} ; 
\\ 
b = \sum_Q b_Q, \quad \supp (b_Q) \subset Q, \quad 
\|b_Q\|_{L^1} \leq c\lambda |Q|. 
\end{gather}    
\par
Let a polynomial $P$ be as in Proposition 1. We assume as we may that 
$M \geq 1$ 
as in the outline of the proof of Theorem 1 in \S 2.  
We write $P$ as in (1.5). 
Then, let   
$q(y) = \sum_{|\beta|\leq L}c_{\beta}y^{\beta}$ be the coefficient of $x_1^M$. 
 By a rotation of coordinates and a normalization, to prove Proposition 1 we may assume  $\max_{|\beta|= L}|c_{\beta}| = 1$ (see \cite[p.\ 151]{1} and 
Sublemma 2 in \S 6).  

We take a non-negative $\varphi \in C^{\infty}_0(\bold R^n)$ such that 
$$\text{supp}(\varphi) \subset \{1/2 \leq |x| \leq 2\}, \qquad  
\sum_{j = 0}^{\infty} \varphi (2^{-j}x) = 1 \quad \text{if} \quad  |x| \geq 1.$$

Put $K_j(x,y) = \varphi (2^{-j}(x - y))K_{\infty}(x,y)$, where $K_{\infty}(x,y) = e^{iP(x,y)}K_{\infty}(x - y)$ ($K_{\infty}(x)$ is as in \S 2) and decompose $K_{\infty}(x,y)$ as $K_{\infty}(x,y) = \sum_{j = 0}^{\infty}K_j(x,y)$.

Define 
$$V_j(f)(x) = \int K_j(x,y)f(y)\, dy \qquad \text{for} \qquad j \geq 0 $$
and put
$$V(f)(x) = \sum_{j=1}^{\infty}V_j(f)(x).$$
Then $T_{\infty} = V_0 + V$.  We have only to deal with  $V$ since  we easily see that $V_0$ is bounded on $L_w^1$ ($w^u \in A_1$). 

We set (see  \cite{3, 4}) 
$$B_i = \sum_{|Q|=2^{in}}b_Q \quad (i \geq 1), \qquad B_0 = \sum_{|Q|\leq 1}b_Q .$$

Put $\mathcal U = \cup \tilde{Q}$, where $\tilde{Q}$ denotes the cube with the same center as $Q$ and with sidelength $100$ times that of $Q$. 
(Throughout this note we consider the cubes with sides parallel to the 
 coordinate axes.)       
\par
When $x \in \bold R^n\setminus \mathcal U$, we observe that  

\begin{multline}   \label{tag 4.6} 
V(b)(x) = V\left(\sum_{i\geq 0}B_i\right)(x) 
\\  
= \sum_{i\geq 0}\sum_{j\geq 1}\int K_j(x,y)B_i(y)\, dy 
= \sum_{i\geq 0}\sum_{j\geq i+1}\int K_j(x,y)B_i(y)\, dy  
\\  
= \sum_{s\geq 1}\sum_{j\geq s}\int K_j(x,y)B_{j-s}(y)\, dy 
= \sum_{s\geq 1}\sum_{j\geq s}V_j(B_{j-s})(x). 
\end{multline}   
\par 
To prove Proposition 1 we need the following results (Lemmas 2, 3 and 4). 

\begin{lemma}  Suppose $w \in A_1$.  
Let $\{L_j\}_{j\geq 1}$ be a family of kernels satisfying 
$$
\supp(L_j) \subset \{2^{j-6}\leq |x|\leq 2^{j+6}\}, 
\quad 
|L_j(x)|\leq c_1|x|^{-n},
\quad 
|\nabla L_j(x)|\leq c_2|x|^{-n-1}.
$$
Let
$$G_j(f)(x) = 
\int_{\text{${\bold R}^n$}} e^{iP(x,y)}L_j(x - y)f(y)\, dy. $$ 
Put 
$$E_\lambda^s=\left\{x\in \bold R^n : 
\left|\sum_{j\geq s}G_j(B_{j-s})(x)\right|>\lambda\right\}.$$
Then 
there exists  $\epsilon, \eta > 0$ such that, for any positive integer $s$, 
$$w\left(E_{\lambda c_\eta2^{-\eta s}}^s\right) \leq 
c2^{-\epsilon s}\lambda^{-1} 
\|f\|_{L^1_w},   $$
where $c_\eta$ is a positive constant satisfying $\sum_{s=1}^\infty
c_\eta2^{-\eta s/2}=1$.   
\end{lemma} 
\par 
We shall prove this in \S 5.
\begin{lemma}
Let $L_j$ and $G_j$ be as in Lemma $2$.  
 Then, for $j \geq 1$, 
$$\|G_j\|_{2} \leq c2^{-j\epsilon}\quad \text{for some $\epsilon>0$,} $$
where $\|G_j\|_{2}$ denotes the operator norm on $L^2$.
\end{lemma} 
This follows from Ricci-Stein \cite{5}. See also \cite{8} for 
an alternative proof.

\begin{lemma}If $w^u\in A_1$, then the operator $V$ is bounded on 
$L^2_w$.        
\end{lemma} 
\begin{proof} Let 
$$N_j(x) = \varphi (2^{-j}x)K(x),\quad L_j(x)=N_j*\psi_{2^{-j+\delta j}}(x) 
\quad (\delta>0),$$
where $\psi\in C^\infty(\bold R^n)$ which is supported in $\{|x|<2^{-10}\}$ 
and satisfying $\int \psi =1$.  Then $L_j$ satisfies all the conditions of 
Lemma 2  with $c_1=c2^{n\delta j}$, $c_2= c2^{(n+1)\delta j}$, and we find   
\begin{gather} 
\|L_j\|_{L^1}\leq c\,C_1 , \label{tag 4.7}
\\
\|L_j\|_{L^r}\leq c\,C_r2^{-jn/u}. \label{tag 4.8}    
\end{gather} 
Put 
$$R_j(x)=N_j(x)-L_j(x)=\int \left(N_j(x)-N_j(x-y)\right)
\psi_{2^{-j+\delta j}}(y) \,dy.$$   
Then, it is easy to see that  
\begin{gather}  
\|R_j\|_{L^1}\leq c\,\omega_1(2^{-\delta j})+ c2^{-\delta j}
\leq c\omega_r(2^{-\delta j})+ c2^{-\delta j}, \label{tag 4.9} 
\\
\|R_j\|_{L^r}\leq c(\omega_r(2^{-\delta j})+ c2^{-\delta j})2^{-jn/u}. 
\label{tag 4.10} 
\end{gather}
Put
$$ U_j(f)(x) = 
\int_{\text{${\bold R}^n$}} e^{iP(x,y)}L_j(x - y)f(y)\, dy,
\qquad 
W_j(f)(x) = 
\int_{\text{${\bold R}^n$}} e^{iP(x,y)}R_j(x - y)f(y)\, dy.
$$
\par
First we estimate $U_j$. By H\"older's inequality and (4.7), (4.8) we have 
\begin{align} \label{tag 4.11} 
\|U_j(f)\|_{L^2_w}^2 &\leq c\int
\left(\int |L_j(x-y)|w(x)\,dx\right)|f(y)|^2\,dy  
\\   
&\leq c\int|f(y)|^2M_{u}(w)(y)\,dy.    \notag 
\end{align}  
On the other hand, if $\delta$ is small enough, by Lemma 3 
\begin{equation}\label{tag 4.12} 
\|U_j(f)\|_{L^2}^2\leq c2^{-\epsilon j}\|f\|_2^2 
\quad \text{for some $\epsilon>0$}.  
\end{equation} 
Interpolating between the estimates (4.11) and (4.12), we get   
$$\|U_j(f)\|_{L^2_{w^\theta}}^2
\leq c2^{-\epsilon (1-\theta)j}\int|f(y)|^2M_{u}(w)(y)^\theta\,dy, $$
for $\theta\in (0,1)$.  
Substituting $w^{1/\theta}$ for $w$, we have 
\begin{equation}   \label{tag 4.13} 
\|U_j(f)\|_{L^2_{w}}^2
\leq c2^{-\epsilon (s-u)j/s}\int|f(y)|^2M_{s}(w)(y)\,dy 
\quad \text{for all $s>u$}. 
\end{equation}  
\par 
Next we estimate $W_j$.  By H\"older's inequality and (4.9), (4.10)  
\begin{align} 
\|W_j(f)\|_{L^2_w}^2&\leq c(\omega_r(2^{-\delta j})+ c2^{-\delta j})\int
\left(\int |R_j(x-y)|w(x)\,dx\right)|f(y)|^2\,dy \label{tag 4.14}   
\\
&\leq c(\omega_r(2^{-\delta j})+ c2^{-\delta j})^2\int|f(y)|^2M_{u}(w)(y)\,dy. 
\notag 
\end{align}
By (4.13) and (4.14), for all $s>u$,  
$$
\|V(f)\|_{L^2_w}
\leq c\sum_{j\geq 1}(\omega_r(2^{-\delta j})+2^{-\delta j}+
2^{-\epsilon (s-u)j/(2s)})\|f\|_{L^2_{M_{s}(w)}} 
\leq c_s\|f\|_{L^2_{M_{s}(w)}}. 
$$
From this we get the conclusion of Lemma 4, since $w^s\in A_1$ for 
some $s>u$.         
\end{proof} 

 Using these results, we can prove Proposition 1.   
   Let $N_j$ and $\psi$ be as in the proof of Lemma 4. 
For a positive integer $s$ let
$$ L_j^{(s)}(x)=N_j*\psi_{2^{-j+\delta s}}(x) \quad (\delta>0).$$ 
Put 
$$R_j^{(s)}(x)=N_j(x)-L_j^{(s)}(x)=\int \left(N_j(x)-N_j(x-y)\right)
\psi_{2^{-j+\delta s}}(y) \,dy.$$   Then $L_j^{(s)}$ is supported in 
$\{2^{j-6}\leq |x|\leq 2^{j+6}\}$ and satisfies 
$$
|L_j^{(s)}(x)|\leq c2^{n\delta s}|x|^{-n},
\qquad   
|\nabla L_j^{(s)}(x)|\leq c2^{(n+1)\delta s}|x|^{-n-1}.
$$
Set 
$$ U_j^{(s)}(f)(x) = 
\int_{\text{${\bold R}^n$}} e^{iP(x,y)}L_j^{(s)}(x - y)f(y)\, dy,
\ \ 
W_j^{(s)}(f)(x) = 
\int_{\text{${\bold R}^n$}} e^{iP(x,y)}R_j^{(s)}(x - y)f(y)\, dy.
$$
Put 
$$F_\lambda^{s}=\left\{x\in \bold R^n : 
\left|\sum_{j\geq s}U_j^{(s)}(B_{j-s})(x)\right|>\lambda\right\}.$$
Then, if $(n+1)\delta<\eta/2$ by Lemma 2 
\begin{equation}  \label{tag 4.15} 
w\left(F_{c_\eta 2^{-\eta s/2}\lambda}^{s}\right) \leq 
c2^{-\epsilon s}\lambda^{-1} \|f\|_{L^1_w}, 
\end{equation} 
where  $\epsilon$, $\eta$ and $c_\eta$ are as in Lemma 2.     
Since $\sum_{s=1}^\infty c_\eta2^{-\eta s/2}=1$, we have 
$$\left\{x\in \bold R^n : 
\left|\sum_{s\geq 1}\sum_{j\geq s}U_j^{(s)}(B_{j-s})(x)\right|>\lambda
\right\} \subset \bigcup_{s\geq 1}F_{c_\eta 2^{-\eta s/2}\lambda}^s.$$
Thus by (4.15) 
\begin{align}  
w\left(\left\{x\in \bold R^n : 
\left|\sum_{s\geq 1}\sum_{j\geq s}U_j^{(s)}(B_{j-s})(x)\right|>\lambda
\right\}\right) 
&\leq \sum_{s\geq 1}w\left(F_{c_\eta 2^{-\eta s/2}\lambda}^s\right)
\label{tag 4.16}  
\\
&\leq c\lambda^{-1}\|f\|_{L^1_w}.   \notag 
\end{align} 
\par 
Since 
$$\|R_j^{(s)}\|_{L^r}\leq c(\omega_r(2^{-\delta s})+ 2^{-\delta s})2^{-jn/u},$$
by H\"older's inequality and the condition that $w^u\in A_1$ we find  
$$\left\|\sum_{j\geq s} W_j^{(s)}(B_{j-s})\right\|_{L^1_w}
\leq c\left(\omega_r(2^{-\delta s})+2^{-\delta s}\right)\|f\|_{L^1_w}.$$
Thus, by Chebyshev's inequality we have 
\begin{multline}  \label{tag 4.17}  
w\left(\left\{x\in \bold R^n : 
\left|\sum_{s\geq 1}\sum_{j\geq s}W_j^{(s)}(B_{j-s})(x)\right|>\lambda
\right\}\right)\\
\leq c\left(\sum_{s\geq 1}\left(\omega_r(2^{-\delta s})+2^{-\delta s}\right)
\right)\lambda^{-1}\|f\|_{L^1_w}.  
\end{multline} 
By (4.6), (4.16) and (4.17) we have 
\begin{equation} \label{tag 4.18}  
w\left(\left\{x\in \bold R^n\setminus \mathcal U : |V(b)(x)|> 2\lambda 
\right\}\right) \leq c\lambda^{-1}\|f\|_{L^1_w}. 
\end{equation} 
\par 
By (4.3) we see that 
\begin{equation} \label{tag 4.19}  
w(\mathcal U) \leq c_w\lambda^{-1}\|f\|_{L^1_w}. 
\end{equation} 
By Lemma 4 and (4.4) 
\begin{equation} \label{tag 4.20}   
w\left(\left\{x \in \bold R^n : |V(g)(x)| > \lambda  \right\}\right) 
 \leq c\lambda^{-1}\|f\|_{L^1_w}. 
\end{equation} 
\par
Combining (4.18), (4.19) and (4.20), we conclude the proof of Proposition 1.

\section{Proof of Lemma $2$}

In this section we shall prove Lemma 2 in \S 4. 
 For $k, m \geq 1$, put
\begin{equation}  \label{tag 5.1}    
H_{km}(x,y) = \int e^{-iP(z,x)+iP(z,y)}\overline{L}_k(z - x)L_m(z - y)\, dz.  
\end{equation}
Then $G_k^*G_m(f)(x) = \int H_{km}(x,y)f(y)\, dy$, where  
$G_k^*$ denotes the adjoint of $G_k$.

\begin{lemma}  Let $k \geq m \geq 1$. Then, 
$H_{km}(x,y) = 0 $ unless $|x - y| \leq 2^{k+7};$ and 
\begin{enumerate}
\item[(1)] 
 \qquad  \qquad $|H_{km}(x,y)| \leq c2^{-kn}$,    
\item [(2)]  
\qquad \qquad  
$|H_{km}(x,y)| \leq c2^{-kn}2^{-m}|q(x) - q(y)|^{-1/M}.$
\end{enumerate} 
\end{lemma}

\begin{proof}  We prove only the estimate of (2) since the other assertions 
immediately follow from the definition of $H_{km}$ in (5.1).  
 We first note that
$$(\partial /\partial z_1)^M(P(z,x) - P(z,y)) = M!(q(x) - q(y)).$$ 
Hence, from van der Corput's lemma it follows that
$$\left|\int_a^b e^{i(P(z,x) - P(z,y))}\, dz_1\right| 
\leq c|q(x) - q(y)|^{-1/M},$$ for any $a$ and $b$ (see \cite[ p.\ 152]{1}).  
\par 
Therefore by integration by parts in variable $z_1$ in the formula of (5.1)  
 we get the conclusion.  
\end{proof} 
\par 
For the rest of this note, we denote by $P(x)$ a real-valued polynomial on 
$\bold R^n$.    
\begin{definition}
For a polynomial $P(x) = \sum_{|\alpha|\leq N}a_{\alpha}x^{\alpha}$ of degree $N$, define
$$\|P\| = \max_{|\alpha|= N}|a_{\alpha}|.$$
\end{definition} 
\begin{definition}
For a polynomial $P$ and $\beta > 0$, let
$$\mathcal R(P,\beta) = \{x \in \bold R^n : |P(x)| \leq \beta \}.$$ 
\end{definition} 

Let $d(E,F)$ denote the distance between sets $E$ and $F$.  
We now state a geometrical lemma for polynomials, which will be proved in \S 6.

\begin{lemma}  Let $k$, $m $ be integers  such that $k \geq m $.  Suppose 
$N \geq 1$.  Then, for any polynomial $P$ of degree $N$ satisfying $\|P\| = 1$ 
and for any $\gamma >0$,  there exists a positive constant $C_{n,N,\gamma}$ 
depending only on $n$,  $N$ and $\gamma$  such that 
$$\left|\left\{x \in B(a,2^k) : d\left(x, \mathcal R(P,2^{Nm})\right) \leq \gamma 
2^m \right\}\right| \leq C_{n,N, \gamma}2^{(n-1)k}2^m $$ 
uniformly in $a \in \bold R^n$.   
\end{lemma}

Let $\lambda > 0$ and let $\{\mathcal B_j\}_{j\geq 0}$ be a family of measurable functions such that 
\begin{equation} \label{tag 5.2} 
\int_Q |\mathcal B_j| \leq \lambda |Q| 
\end{equation}
for all cubes $Q$ in $\bold R^n$ with sidelength $\ell (Q) = 2^j$.   
\par 
Then we have the following. 

\begin{lemma}  Let the kernels $H_{ji}$ be as in Lemma $5$.  Then, 
we can find a constant $c$ such that 
$$\sum_{i=s}^j\sup_{x \in \text{${\bold R}^n$}} \left|\int \mathcal B_{i-s}(y)H_{ji}(x,y)\, dy \right| \leq c\lambda 2^{-s}$$for all integers $j$ and $s$ such 
that $0 < s \leq j$.
\end{lemma} 

\begin{definition} 
For $m \in \bold Z$ (the set of all integers), let $\mathcal D_m$ be the family of all closed dyadic cubes $Q$ with sidelength $\ell (Q) = 2^m$.  
\end{definition} 

\begin{proof}[Proof of Lemma $7$]  Fix $x \in \bold R^n$.  Let 
$$\mathcal F =  \left\{Q \in \mathcal D_{i-s} : Q \cap B(x,2^{j+2}) 
\ne \emptyset 
\right\} \quad (0 < s \leq i \leq j).$$ 
Then clearly  $\sum_{Q \in \mathcal F}|Q| \leq c2^{jn}$.  

Decompose $\mathcal F = \mathcal F_0 \cup \mathcal F_1$, where
$$\mathcal F_0 =  \left\{Q \in \mathcal F : Q \cap \mathcal R(q(\cdot) - q(x),2^{L(i-s)}) \ne \emptyset \right\}$$ 
and $\mathcal F_1 = \mathcal F \setminus \mathcal F_0$.  Then by Lemma 6 we have\begin{equation} \label{tag 5.3}  
\sum_{Q \in \mathcal F_0}|Q| \leq c2^{(n-1)j}2^{i-s}. 
\end{equation} 
\par 
By Lemma 5 (1), (5.2) and (5.3), we see that
\begin{multline}  \label{tag 5.4}  
\sum_{Q \in \mathcal F_0}\int_Q\left|\mathcal B_{i-s}(y)H_{ji}(x,y)\right|\, dy  
\leq c2^{-jn}\sum_{Q \in \mathcal F_0}\int_Q\left|\mathcal B_{i-s}(y)\right|\, dy 
 \\
\leq c2^{-jn}\lambda \sum_{Q \in \mathcal F_0} |Q| 
\leq c2^{-jn}\lambda 2^{(n-1)j}2^{i-s} 
= c\lambda 2^{i-j-s}. 
\end{multline}

Next, by Lemma 5 (2), (5.2) and the estimate $\sum_{Q \in \mathcal F_1}|Q| 
\leq c2^{jn}$,  we have   
\begin{multline}   \label{tag 5.5} 
\sum_{Q \in \mathcal F_1}\int_Q\left|\mathcal B_{i-s}(y)H_{ji}(x,y)\right|\, dy  
\leq c2^{-jn}2^{-i}2^{-L(i-s)/M}\sum_{Q \in \mathcal F_1}\int_Q\left|\mathcal B_{i-s}(y)\right|\, dy 
 \\
 \leq c2^{-jn}2^{-i}2^{-L(i-s)/M}\lambda \sum_{Q \in \mathcal F_1} |Q| 
 \leq c\lambda 2^{-i}2^{-L(i-s)/M}. 
\end{multline}   

From  (5.4) and (5.5) it follows that
\begin{multline*}   
 \int \left|\mathcal B_{i-s}(y)H_{ji}(x,y)\right|\, dy  = 
\sum_{Q \in \mathcal F}\int_Q\left|\mathcal B_{i-s}(y)H_{ji}(x,y)\right|\, dy  
\\
= \sum_{\nu=0}^1\sum_{Q \in \mathcal F_{\nu}}\int_Q\left|\mathcal B_{i-s}(y)H_{ji}(x,y)\right|\, dy  
\leq c\lambda \left(2^{i-j-s} +  2^{-i}2^{-L(i-s)/M}\right). 
\end{multline*}   
Thus we see that
$$
\sum_{i=s}^j\sup_{x \in \text{${\bold R}^n$}} 
\int \left|\mathcal B_{i-s}(y)H_{ji}(x,y)\right|\, dy  
\leq c\lambda \sum_{i=s}^j\left(2^{i-j-s} +  2^{-i}2^{-L(i-s)/M}\right) 
\leq c\lambda 2^{-s}. 
 $$
This completes the proof of Lemma 7.
\end{proof}

By Lemma 7 we readily get the following.

\begin{lemma}  Let $\{\mathcal B_j\}_{j\geq 0}$ be as in Lemma $7$.  Suppose 
$\sum_{j\geq 0}\|\mathcal B_{j}\|_{L^1} < \infty$.  Let $G_j$ be as in 
Lemma $2$. 
Then, for any positive integer $s$, we have 
$$\left\|\sum_{j\geq s}G_j(\mathcal B_{j-s})\right\|_{L^2}^2 
\leq c\lambda 2^{-s}\sum_{j\geq 0}\|\mathcal B_{j}\|_{L^1}.  
$$ 
\end{lemma} 

\begin{proof} Let $\langle \cdot,\cdot \rangle$ denote the inner product in 
$L^2$.  Using Lemma 7, we see that
\begin{multline*} 
\left\|\sum_{j\geq s}G_j(\mathcal B_{j-s})\right\|_{L^2}^2 
\leq 2\sum_{j\geq s}\sum_{i=s}^j\left|\langle 
G_j(\mathcal B_{j-s}), G_i(\mathcal B_{i-s})\rangle \right| 
\leq 2\sum_{j\geq s}\sum_{i=s}^j\left|\langle \mathcal B_{j-s}, G_j^*G_i
(\mathcal B_{i-s})\rangle \right|
\\  
\leq 2\sum_{j\geq s}\sum_{i=s}^j\|\mathcal B_{j-s}\|_{L^1}
\|G_j^*G_i(\mathcal B_{i-s})\|_{L^\infty}
\leq c\lambda 2^{-s}\sum_{j\geq s}\|\mathcal B_{j-s}\|_{L^1}.
\end{multline*}   
This completes the proof of Lemma 8.
\end{proof}

\begin{definition} 
For each $j \geq 0$, let $\mathcal G_j$  be a family of non-overlapping closed 
dyadic cubes $Q$ such that $\ell (Q) \leq 2^j$.  We suppose that if $Q \in \mathcal G_j$, $R \in \mathcal G_k$ and $j \ne k$, then $Q$ and $R$ are non-overlapping and 
that $\sum_{j\geq 0}\sum_{Q\in \mathcal G_j}|Q| < \infty$.   Put $\mathcal G = 
\cup_{j\geq 0}\mathcal G_j$. 
\end{definition} 

Let $\lambda > 0$.  To each $Q \in \mathcal G $ we associate $f_Q \in L^1$  
such that
$$ \int |f_Q| \leq \lambda |Q|, \qquad \text{supp}(f_Q) \subset Q.$$   
We define  $\mathcal A_i = \sum_{Q \in \mathcal G_i}f_Q$.

\begin{lemma}  Let $G_j$ be as in Lemma $2$ and 
let $v$ be a locally integrable positive function. 
 Then for a positive integer $s$ we have 
$$\left\|\sum_{j\geq s}G_j(\mathcal A_{j-s})\right\|_{L^1_v} \leq c\lambda 
\sum_{Q \in \mathcal G}|Q| \inf_{Q} M(v), $$
where $\inf_QM(v) = \inf_{x \in Q}M(v)(x)$. 
\end{lemma} 
\begin{proof} We easily see that 
\begin{multline*}  
\left\|\sum_{j\geq s}
G_j\left(\mathcal A_{j-s}\right)\right\|_{L^1_v} \leq \sum_j
\int |\mathcal A_{j-s}(y)|\left(\int |L_j(x-y)|v(x)\,dx\right)\,dy
\\
\leq \sum_j\sum_{Q\in \mathcal G_{j-s}}\int |f_Q(y)|\inf_{z\in Q}M(v)(z)\,dy
 \leq c\sum_{Q \in \mathcal G} \lambda|Q|\inf_QM(v).
\end{multline*} 
\end{proof}

We prove Lemma 2 by the estimates of Lemma 8 and Lemma 9.  
We slightly modify the interpolation argument of \cite{9}. 

\begin{lemma}  Let $\mathcal F$ denote the family of dyadic cubes arising 
from the Calder\'{o}n-Zygmund decomposition in \S $4$. Define a set 
$E_\lambda^s$ as in Lemma $2$. 
Then, for all $t > 0$, we have  
\begin{equation}\label{tag 5.6}  
\int_{E_\lambda^s} \min(v(x),t)\, dx 
\leq c \sum_{Q \in \mathcal F}|Q| \min\left(t2^{-s},\inf_{Q} M(v)\right), 
\end{equation} 
where $s$ is a positive integer and $v$ is a locally integrable positive 
function.
\end{lemma} 

\begin{proof} For $t>0$, set  
$\mathcal F_t = \{Q \in \mathcal F : \inf_Q M(v) < t2^{- s} \}$  
and $\mathcal F_t^* =  \mathcal F \setminus \mathcal F_t$.  Put    
$$B'_j =  \sum\limits_{\substack{\ell(Q)=2^j \\ Q \in \mathcal F_t}} b_Q,   
\quad 
 B^{\prime\prime}_j = \sum\limits_{\substack{\ell(Q)=2^j \\ Q \in 
\mathcal F_t^* }} b_Q 
 \ \ (j\geq 1); \qquad B'_0 = \sum\limits_{\substack{|Q|\leq 1 \\ Q \in 
\mathcal F_t}} b_Q, 
 \quad 
B^{\prime\prime}_0 = \sum\limits_{\substack{|Q|\leq 1 \\ Q \in \mathcal F_t^* 
}} b_Q.$$
Define  
$$ 
E'_\lambda=\left\{  \left|\sum_{j\geq s}
G_j\left(B'_{j-s}\right)\right|>\lambda\right\}, 
\quad   
E^{\prime\prime}_\lambda=\left\{ \left|\sum_{j\geq s}
G_j\left(B^{\prime\prime}_{j-s}\right)\right| >\lambda\right\}.
$$
  Then we find 
$E_\lambda^s \subset E'_{\lambda/2}\cup E^{\prime\prime}_{\lambda/2}$,
since $B_j = B'_j + B^{\prime\prime}_j$,  
and so 
\begin{align*}  
\int_{E_\lambda^s} \min(v(x),t)\, dx  
&\leq \int_{E'_{\lambda/2}}\min(v(x),t)\, dx 
+\int_{E^{\prime\prime}_{\lambda/2}} \min(v(x),t)\, dx  
\\
&\leq \int_{E'_{\lambda/2}} v(x)\, dx  +\int_{E^{\prime\prime}_{\lambda/2}}
t\, dx  =:  I + II.
\end{align*} 
\par 
By Lemma 9 with $\mathcal A_j = cB'_j$, we get
$$ I \leq c\sum_{Q \in \mathcal F_t}|Q| \inf_{Q} M(v) 
= c\sum_{Q \in \mathcal F_t}|Q| 
\min\left(t2^{-s},\inf_{Q} M(v)\right). $$
By Lemma 8 with $\mathcal B_j = cB^{\prime\prime}_j$, we have 
$$ 
 II \leq c t2^{- s}\sum_{Q \in \mathcal F_t^*}|Q| 
=  c \sum_{Q \in \mathcal F_t^*}|Q| \min\left(t2^{-s},\inf_{Q} 
M(v)\right). 
$$
Combining the estimates for $I$ and $II$, we conclude the proof of Lemma 10. 
\end{proof} 
\par  
Now we finish the proof of Lemma 2.  Multiplying both sides of the inequality 
(5.6) by $t^{-\theta}$ ($\theta \in (0, 1)$), then integrating them on 
$(0, \infty)$ with respect to the measure $dt/t$,  and using
$$\int_0^{\infty}\min(A,t)t^{-\theta}\, \frac{dt}{t} = c_{\theta}A^{1-\theta} 
\quad (A > 0),$$
we get 
\begin{multline} \label{tag 5.7}  
\int_{E_\lambda^s} v(x)^{1-\theta}\, dx 
\leq c \sum_{Q \in \mathcal F}|Q|2^{-\theta s} \inf_{Q} M(v)^{1-\theta}
 \\
\leq  c\lambda^{-1} 2^{-\theta s}\sum_{Q \in \mathcal F}\inf_{Q} M(v)^{1-\theta}  
\int_Q|f(x)|\, dx 
\leq  c\lambda^{-1} 2^{-\theta s}\int |f(x)|M(v)(x)^{1-\theta}\, dx,
\end{multline}  
where the second inequality follows from (4.2).  
\par 
If $w \in A_1$, then $w^{1+\delta} \in A_1$ for some $\delta > 0$; 
so substituting  $w^{1+\delta}$ for $v$ and taking $\theta$ such that 
$1 - \theta = (1+\delta)^{-1}$ in (5.7), we get 
\begin{equation}\label{tag 5.8}  
 w\left(E_{\lambda}^{s}\right) 
\leq c\lambda^{-1}2^{- s\delta/(1+\delta)}\|f\|_{L^1_w}. 
\end{equation} 
Checking the constants appearing in the proof of (5.8) and replacing 
$L_j$ by $c2^{\eta s}L_j$, we get the desired estimate of Lemma 2. 

\section{Proof of Lemma $6$}

Our proof is an application of the method for the proof of 
\cite[{\sc Lemma} 4.1]{1}.   
We use some tools and results given in \cite{1}.

\begin{definition}  Suppose  $n \geq 2$.   Let 
$$S_m = \left\{Q_m + (0, 0, \dots , 0, j) : j \in \bold Z \right\},$$
where $m = (m_1, m_2, \dots , m_{n-1}) \in \bold Z^{n-1}$ and  $Q_m = [0, 1]^n 
+ (m_1, m_2, \dots , m_{n-1}, 0)$.   
We call $S_m$ a strip. 
\end{definition} 

\begin{definition} Suppose  $n \geq 2$. For $m \in \bold Z^{n-1}$, 
we define   
$$I_m = \left\{Q_m + (0, 0, \dots , 0, j) :  j_1 <  j < j_2 \right\}, $$
where $j_1$, $j_2 \in \bold Z \cup \{-\infty, \infty \}$ and $Q_m$ is as in 
Definition 5.  We call $I_m$ an interval.
\end{definition}

\begin{definition} For a set $E \subset \bold R^n$, we put
 $$\mathcal N (E) = \left\{x \in \bold R^n : d(x,E) \leq 1 \right\}.$$
\end{definition} 

Let $P$ be a polynomial of degree $N$ as in Lemma 6. We consider 
$\mathcal R(P,\beta)$ for $\beta > 0$ (see Definition 2). 

\begin{lemma} Suppose that $n \geq 2$ and $N \geq 1$.  
There exists a positive integer $C_{n,N}$ depending only on $n$ and $N$ such that for $i = 1, 2, \dots , C_{n,N}$ we can find $U_i \in O(n)$ $($the orthogonal 
group$)$ and families of cubes $J_{m,i} \subset S_m  \, (m \in \bold Z^{n-1})$ 
so that 
\begin{enumerate} 
\item[(1)]    
$\mathcal N(\mathcal R(P,\beta)) \subset \bigcup_{i=1}^{C_{n,N}}U_i(\mathcal L_i),$    where 
$$\mathcal L_i = \cup \left\{Q : Q \in \bigcup_{m\in \text{${\bold Z}^{n-1}$}}
J_{m,i} \right\};$$ 
\item[(2)]   
$\card (J_{m,i}) \leq c$   
for some constant $c$ depending only on $n$, $N$ and $\beta$.
\end{enumerate}
\end{lemma}     

\begin{remark} If Lemma 11 holds, then we have, for any $ \gamma > 0$, 
$$ \{x : d(x,\mathcal R(P,\beta)) \leq \gamma \} \subset \bigcup_{i=1}^{C_{n,N,\gamma}}U_i(\mathcal L_i)$$ 
for some positive integer $C_{n,N,\gamma}$ depending only on $n$, $N$ and 
$\gamma$, where $U_i$ and $\mathcal L_i$ are as in Lemma 11.  
This can be proved by 
considering a finite number of polynomials which are defined by translating  
$P$ and by applying Lemma 11 to each of them. 
(See \cite[p.\ 149]{1}.) 
\end{remark} 

To prove Lemma 11, we need the following results given in  \cite{1}.

\begin{sublemma} Suppose $n \geq 2$. 
For any positive integer $N$, there exists a positive integer $C_{n,N}$ depending only on $n$ and $N$ such that for any strip $S$, any polynomial $P$ of degree $N$ and any $\gamma > 0$
$$\{Q \in S : Q \cap \mathcal R(P,\gamma) \ne \emptyset \}$$
is a union of at most $C_{n,N}$ intervals. {\rm (See {\sc Lemma} 4.2 of 
\cite{1}.)}  
\end{sublemma} 

\begin{sublemma} Suppose  $n \geq 2$.  
For any positive integer $N$, there exist positive  constants 
 $A_{n,N}$ and $B_{n,N}$  depending only on $n$ and $N$ such that 
$$A_{n,N}\|P\| \leq \|P\circ\Xi\| \leq B_{n,N}\|P\|$$
for all polynomial $P$ of degree $N$ and all  $\Xi \in O(n)$, where $P\circ\Xi (x) = P(\Xi x)$.
\end{sublemma} 

\begin{sublemma} Suppose $n \geq 2$.  
For any positive integer $N$, there exists a positive constant $C_{n,N}$ depending only on $n$ and $N$ such that for any polynomial $P$ of degree $N$ we can find  $\Theta \in O(n)$ so that 
$$\min_{1\leq j \leq n}\|D_j(P\circ\Theta)\| \geq  C_{n,N}\|P\circ\Theta\|,$$
where $D_j = \partial / \partial x_j$.
\end{sublemma} 

Now we prove  Lemma 11.  We use induction on the polynomial degree $N$. Let 
$A(N)$ be the assertion of Lemma 11 for polynomials of degree $N$.  

\par 
Proof of $A(1)$.  Let $P(x) = \sum_{i=1}^na_ix_i + b$. 
First, we consider the case $|a_n| = 1$.  Now we show that if $I$ is an interval such that each cube of $I$ intersects $\mathcal R(P,\beta)$,  then $\text{card} (I) \leq c$ for some $c$ depending only on $n$ and $\beta$.  
Let $y \in Q \in I$ satisfy  $|P(y)| \leq \beta$.  We note that  
$$P(y + de_n) - P(y) = da_n \quad \text{for}\quad d \in \bold R,$$
where $e_j$ is the element of $\bold R^n$ whose jth coordinate is $1$ and whose other coordinates are all $0$.  Therefore, if $y + de_n \in Q' \in I$, we see 
that
$$
\inf_{z\in Q'}|P(z)| \geq |P(y + de_n)| - \sum_{i=1}^n|a_i|    
\geq |da_n| - \beta - \sum_{i=1}^n|a_i|  
\geq |d| - \beta - n.
$$
This easily implies that $\card (I) \leq c$.

By this and Sublemma 1, there exists a constant $c$ depending only on $n$ and $\beta$ such that 
$$\card \left( \left\{Q \in S : Q \cap \mathcal R(P,\beta) \ne \emptyset 
\right\}\right)  \leq c$$   for all strips $S$.  

Therefore, if we put 
$$J_m = \left\{Q \in S_m : d(Q,\mathcal R(P,\beta)) \leq 1  \right\},$$
then $\card (J_m) \leq c$ for some $c$ depending only on $n$ and 
$\beta$ ; and $\mathcal N(\mathcal R(P,\beta)) \subset \mathcal L$, where 
$$\mathcal L = \cup \left\{Q : Q \in \bigcup_{m\in \bold Z^{n-1}}J_{m}\right\}.
$$

Next, we consider any polynomial $P$ of degree $1$ such that $\|P\| = 1$.
Then if $P_1(x) =P(Ux)$ for a suitable $U \in O(n)$, we have $D_nP_1 = 1$. Hence, by what we have already proved we get  $\mathcal N(\mathcal R(P_1,\beta)) \subset \mathcal L$.  It follows that $\mathcal N(\mathcal R(P,\beta)) \subset 
U(\mathcal L)$
since $\mathcal N(\mathcal R(P\circ U,\beta)) = 
U^{-1}\mathcal N(\mathcal R(P,\beta))$.  This completes the proof of $A(1)$. 

Now we assume $A(N-1)$   ($N \geq 2$) and prove $A(N)$.    
For a polynomial $P$ of degree $N$ such that $\|P\| = 1$, we take $\Theta \in 
O(n)$ as in Sublemma 3.  Put
$$E_0 =  \mathcal R(P\circ\Theta,\beta)\cap \left(\bigcup_{j=1}^n
\mathcal R(D_j(P\circ\Theta),\beta)\right) ;$$
and for $\kappa =(\kappa_1, \kappa_2, \dots , \kappa_n) \in \{-1, 1\}^n$  put
$$E_{\kappa} = \left\{x \in \mathcal R(P\circ\Theta,\beta) : \kappa_jD_j(P\circ
\Theta)(x) > \beta \quad \text{for}\quad j = 1, 2, \dots , n \right\}.$$
Then
$$\mathcal R(P\circ\Theta,\beta) = E_0 \cup \left(
\bigcup_{\kappa \in \{-1, 1\}^n}E_{\kappa}\right)
 $$
and so 
\begin{equation}\label{tag 6.1} 
\mathcal N(\mathcal R(P\circ\Theta,\beta)) = \mathcal N(E_0) \cup \left(\bigcup_{\kappa \in \{-1, 1\}^n}\mathcal N(E_{\kappa})\right). 
\end{equation}
We separately treat the $2^n +1$ sets of the right hand side. 

First, clearly 
\begin{equation} \label{tag 6.2}  
\mathcal N\left(E_0\right) \subset \bigcup_{j=1}^n \mathcal N\left(\mathcal R\left(D_j
(P\circ\Theta),\beta\right)\right).
\end{equation} 
Since $C_j = \|D_j(P\circ\Theta)\| \sim 1$ ( this means that $c^{-1} \leq \|D_j(P\circ\Theta)\| \leq c$ for some $c > 1$ depending only on $n$ and $N$) and 
$\mathcal R(D_j(P\circ\Theta),\beta) = 
\mathcal R(C_j^{-1}D_j(P\circ\Theta),C_j^{-1}\beta),$   we can apply the 
induction hypothesis $A(N-1)$ to the right hand side of (6.2). 

Next, we fix $\kappa$ and consider $\mathcal N(E_{\kappa})$.  Take $O_{\kappa} 
\in O(n)$ such that 
$O_{\kappa}(e_n) = n^{-1/2}\kappa$.  Define
$$\mathcal D_0^* = \mathcal D_0 \setminus \left\{Q \in \mathcal D_0 : \left(\bigcup_{j=1}^n
\mathcal R((D_j(P\circ\Theta))\circ O_{\kappa},\beta)\right) \cap Q \ne 
\emptyset \right\}.$$

Since $\|(D_j(P\circ\Theta))\circ O_{\kappa}\| \sim 1$ by Sublemmas 2 and 3, we can apply the hypothesis $A(N-1)$ along with Remark 1 to
$$G = \cup \left\{Q \in \mathcal D_0 : \left(\bigcup_{j=1}^n\mathcal R((D_j(P\circ\Theta))\circ O_{\kappa},\beta)\right) \cap Q \ne \emptyset \right\}$$
to get 
\begin{equation} \label{tag 6.3}  
\mathcal N(G) \subset \cup_i U_i'(\mathcal L_i') 
\end{equation} 
for some $U_i' \in O(n)$ and for some $\mathcal L_i'$ such that  
$$\mathcal L_i' = \cup \left\{Q : Q \in \bigcup_{m\in \text{${\bold Z}^{n-1}$}}
J'_{m,i} \right\}$$ 
for some $J'_{m,i}$ ($\subset S_m$) satisfying $\card (J'_{m,i}) \leq c$.
\par 
Therefore we have only to consider $O_{\kappa}^{-1}(E_{\kappa})\cap \left(\cup 
\mathcal D_0^*\right)$.  First, we note that if $O_{\kappa}^{-1}(E_{\kappa})$ intersects 
$Q$,  $Q \in  \mathcal D_0^*$, then  
\begin{equation} \label{tag 6.4} 
\min_{1\leq j \leq n} \kappa_jD_j(P\circ\Theta)(O_{\kappa}y) > \beta \quad \text{for all} \quad y \in Q.  
\end{equation} 

\par 
This can be seen as follows. Suppose that there are $j_0 $ and $y_0 \in Q$ such that  $\kappa_{j_0}D_{j_0}(P\circ\Theta)(O_{\kappa}y_0) \leq \beta$. Then, since we have   $\kappa_{j_0}D_{j_0}(P\circ\Theta)(O_{\kappa}x) > \beta$ for some $x \in Q$, by the intermediate value theorem we can find $z \in Q$ such that $|D_{j_0}(P\circ\Theta)(O_{\kappa}z)| \leq \beta$.  This contradicts the fact that $Q \in \mathcal D_0^*$.
\par 
By (6.4) we have
\begin{multline} \label{tag 6.5} 
O_{\kappa}^{-1}(E_{\kappa})\cap \left(\cup \mathcal D_0^*\right) \subset \cup 
\biggm\{Q \in \mathcal D_0 : \min_{1\leq j \leq n} \kappa_jD_j(P\circ\Theta)(O_{\kappa}y) > \beta 
\\ 
\quad \text{for all} \quad y \in Q 
\quad \text{and} \quad \mathcal R(P\circ\Theta\circ O_{\kappa},\beta) \cap Q 
\ne \emptyset \biggl\}.
\end{multline}
 \par 
For a strip $S$,  put
\begin{multline*}   
\mathcal E = \biggm\{Q \in  S : \min_{1\leq j \leq n} \kappa_jD_j(P\circ\Theta)(O_{\kappa}y) > \beta \quad \text{for all} \quad y \in Q 
\\
\quad \text{and} \quad \mathcal R(P\circ\Theta\circ O_{\kappa},\beta) \cap Q \ne \emptyset \biggl\}. 
\end{multline*} 
We shall show $\card (\mathcal E) \leq C_{n,N}$.
\par 
We first see that $\mathcal E$ is a union of at most $C_{n,N}$ intervals.
   Put 
\begin{multline*}  
\mathcal E' = \biggm\{Q \in S : \min_{1\leq j \leq n} \left|D_j(P\circ\Theta)(O_{\kappa}y)\right| > \beta \quad \text{for all} \quad y \in Q 
\\
\quad \text{and} \quad \mathcal R(P\circ\Theta\circ O_{\kappa},\beta) \cap Q \ne 
\emptyset \biggl\}. 
\end{multline*} 
Then 
\begin{multline*} 
\mathcal E' = \left(\bigcap_{j=1}^n\left(S\setminus \left\{Q \in S : 
\mathcal R((D_j(P\circ\Theta))\circ O_{\kappa},\beta) \cap Q \ne \emptyset 
\right\}\right)\right)
\\
\cap \left\{Q \in S : \mathcal R(P\circ\Theta\circ O_{\kappa},\beta) \cap Q 
\ne \emptyset \right\}. 
\end{multline*} 
We observe that the complement of a finite union of intervals in a strip $S$ is also a finite union of intervals, and the intersection of finite unions of 
intervals is also  a finite union of intervals.  Hence,  by Sublemma 1 we see that  $\mathcal E'$ is a union of at most $C_{n,N}$ intervals: $\mathcal E' = 
\cup_i J_i$.

Take any $J_i$.  Then by the intermediate value theorem we have  either
$$\min_{1\leq j \leq n} \kappa_jD_j(P\circ\Theta)(O_{\kappa}y) > \beta \quad \text{for all} \quad y \in \cup \left\{Q : Q \in J_i\right\}$$
or
$$\min_{1\leq j \leq n} \kappa_jD_j(P\circ\Theta)(O_{\kappa}y) < -\beta \quad \text{for all} \quad y \in \cup \left\{Q : Q \in J_i\right\}.$$
Thus $\mathcal E$ is a union of a subfamily $\{I_i\}$ of $\{J_i\}$ : $\mathcal E = 
\cup_i I_i$. 

Let $I$ be any interval in  $\{I_i\}$.  We need the following (see 
\cite[p.\ 151]{1}).  

\begin{sublemma}  There exists a constant $c_{n}$ depending only on $n$ such that if $x, y \in I$ and $y_n - x_n \geq c_n$, then  
$$y - x = \sum_{i=1}^n\lambda_iO_{\kappa}^{-1}e_i$$ 
for some $\lambda_i \in \bold R$ such that $\kappa_i\lambda_i \geq 3$.
\end{sublemma} 
\begin{proof} We see that
\begin{align*}  
O_{\kappa}(y - x) &= \sum_{i=1}^n(y_i - x_i)O_{\kappa}e_i 
= \sum_{i=1}^{n-1}(y_i - x_i)O_{\kappa}e_i + (y_n - x_n)n^{-1/2}\kappa
\\
&= \sum_{i=1}^n\left(n^{-1/2}(y_n - x_n)\kappa_i + b_i\right)e_i
\end{align*} 
for some $b_i \in \bold R$ such that $|b_i| \leq c$, which is feasible since 
$|y_i - x_i| \leq 1$ for $i = 1, 2, \dots , n - 1$.  This readily implies the 
conclusion. 
\end{proof}  

Put  $Y = P\circ\Theta\circ O_{\kappa}$. Then $\nabla Y(x) = O_{\kappa}^{-1}(\nabla (P\circ\Theta)(O_{\kappa}x))$; so, if $x, y \in I$ and $y_n - x_n \geq c_n$, by Sublemma 4 we have
\begin{multline*} 
Y(y) - Y(x) = \int_0^1 \left\langle y - x , (\nabla Y)(x + t(y - x))
\right\rangle\, dt 
\\
= \int_0^1\sum_{i=1}^n\lambda_i \left\langle O_{\kappa}^{-1}e_i , 
O_{\kappa}^{-1}\left(\nabla (P\circ\Theta)(O_{\kappa}(x + t(y - x)))\right)
\right\rangle\, dt
\\
= \int_0^1\sum_{i=1}^n\lambda_i D_i(P\circ\Theta)(O_{\kappa}(x + t(y - x)))\,  
dt \geq \sum_{i=1}^n\lambda_i\kappa_i\beta \geq 3n\beta > 3\beta,
\end{multline*} 
where $\langle \cdot, \cdot\rangle$ denotes the inner product in $\bold R^n$.  
Since $\mathcal R(Y,\beta) \cap Q \ne \emptyset$ for all $Q \in I$, we can conclude that $\card (I) \leq c_n + 3$.  

Combining the above results, we have $\card (\mathcal E) \leq C_{n,N}$ as 
claimed.    From this and (6.5) we easily see that 
\begin{equation} \label{tag 6.6} 
\mathcal N\left(O_{\kappa}^{-1}(E_{\kappa})\cap 
\left(\cup \mathcal D_0^*\right)\right)
 \subset \mathcal L,  
\end{equation} 
where 
$\mathcal L = \cup \left\{Q : Q \in \bigcup_{m\in \text{${\bold Z}^{n-1}$}}
J_{m} \right\}$ 
for some $J_m \subset S_m$ with $\card (J_{m}) \leq C_{n,N}$. 
\par 
By (6.3) and (6.6) we have 
$$
\mathcal N\left(O_{\kappa}^{-1}(E_{\kappa})\right)  \subset 
\mathcal N(G) \cup \mathcal N\left(O_{\kappa}^{-1}(E_{\kappa})\cap \left(\cup 
\mathcal D_0^*\right)\right)  
\subset \left(\cup_i U_i'(\mathcal L_i')\right) \cup \mathcal L ;
$$
and so, observing $\mathcal N\left(O_{\kappa}^{-1}(E_{\kappa})\right) = 
O_{\kappa}^{-1}\mathcal N\left(E_{\kappa}\right)$,  
\begin{equation}  \label{tag 6.7}  
 \mathcal N\left(E_{\kappa}\right)  \subset  \left(\cup_i O_{\kappa}U_i'
(\mathcal L_i')\right) \cup O_{\kappa}(\mathcal L) . 
\end{equation} 
 \par  
Since $\mathcal N(\mathcal R(P\circ\Theta,\beta)) = \Theta^{-1}\mathcal N(\mathcal R(P,\beta))$, by (6.1), (6.2) with $A(N - 1)$  and (6.7) we get $A(N)$.  This completes the 
proof of Lemma 11.

\begin{proof}[Proof of Lemma $6$]  We see that $\mathcal R(P,2^{Nm}) =
 2^m\mathcal R(\tilde{P},1)$, where 
$$\tilde{P}(x) = 2^{-Nm}P(2^mx). $$  
Note that $\|\tilde{P}\| = 1$.  
(See \cite[p.\ 151]{1}.)  This observation enables us to assume $m = 0$ to prove Lemma 6. Clearly, we may also assume $\gamma = 1$. 

Thus it is sufficient to show, for $k \geq 0$, 
\begin{equation} \label{tag 6.8}  
\left|\left\{x \in B(a,2^k) : d(x, \mathcal R(P,1)) \leq 1 \right\}\right| \leq C_{n,N}2^{(n-1)k}  
\end{equation} 
uniformly in $a \in \bold R^n$.  

If $n = 1$, (6.8) easily follows from  Chanillo-Christ 
\cite[{\sc Lemma} 3.2]{1} (see also \cite{2}).   
Suppose  $n \geq 2$. Then,  (6.8) follows from Lemma 11 with $\beta = 1$ and 
 the obvious estimate: 
$$\left|B(a,2^k) \cap U_i(\mathcal L_i)\right| \leq c2^{(n-1)k},$$
where $U_i(\mathcal L_i)$ is as in Lemma 11.  This completes the proof of 
Lemma 6.
\end{proof}

\end{document}